\pdfoutput=1

\documentclass[12pt]{amsart}%
\usepackage{amssymb}
\usepackage{amsmath}
\usepackage{amsfonts}
\usepackage{graphicx}
\usepackage{version}%
\newtheorem{thm}{Theorem}
\newtheorem{lem}[thm]{Lemma}
\newtheorem{prop}[thm]{Proposition}
\theoremstyle{definition}
\newtheorem{defin}[thm]{Definition}
\newtheorem{example}[thm]{Example}
\theoremstyle{remark}

\numberwithin{equation}{section}
\excludeversion{comment}
\begin{document}
\title[Fast Basins]{Basic Topological Structure of Fast Basins}
\author[M. F. Barnsley]{Michael F. Barnsley}
\address{The Australian National University\\
Canberra, Australia}
\author[K. Le\'{s}niak]{Krzysztof Le\'{s}niak}
\address{Faculty of Mathematics and Computer Science, Nicolaus Copernicus University\\
Toru\'{n}, Poland}

\begin{abstract}
We define fractal continuations and the fast basin of the IFS and investigate
which properties they inherit from the attractor. Some illustrated examples
are provided.

\end{abstract}
\subjclass[2010]{Primary }
\date{\today}
\maketitle

\section{\label{introsec}Introduction}

Fractal continuations, fast basins, and fractal manifolds were introduced in
\cite{FManifolds}; fractal continuation of analytic and other functions was
introduced in \cite{BarnsleyVince}. In this paper we establish some
topological and geometrical properties of continuations and fast basins of
attractors of invertible iterated function systems (IFSs) on complete metric
spaces. Fast basin, attractor, IFS, and other objects, are defined in Section
\ref{defsec}.

Not only are fast basins beautiful objects, illustrated for example in Figure
\ref{fastbasin023}, but also they generalize analytic continuations. For
example, under natural conditions, the fast basin of an analytic fractal is
the same as the analytic continuation of the analytic fractal, when the latter
contains an open subset of an analytic manifold. Fast basins extend the notion
of analytic continuation from the realm of infinitely differentiable objects
to the realm of certain rough, non-differentiable objects.

A contractive IFS comprises a set of contractive transformations and possesses
a unique attractor. An example of an attractor of an IFS is the Sierpi\'{n}ski
triangle $A$ with vertices at $a_{i}\in\mathbb{C}$, the complex plane, where
the IFS comprises three similitudes $z\mapsto f_{i}(z)=\left(  z+a_{i}\right)
/2$ $(i=1,2,3)$. In this case, the attractor is the unique nontrivial compact
set $A\subset\mathbb{C}$ such that $A=f_{1}(A)\cup f_{2}(A)\cup f_{3}(A)$, and
the fast basin is a lattice of copies of $A$, as illustrated in Figure
\ref{manifold1}.
\begin{figure}[ptb]%
\centering
\includegraphics[
natheight=12.386800in,
natwidth=18.560301in,
height=2.1826in,
width=3.2569in
]%
{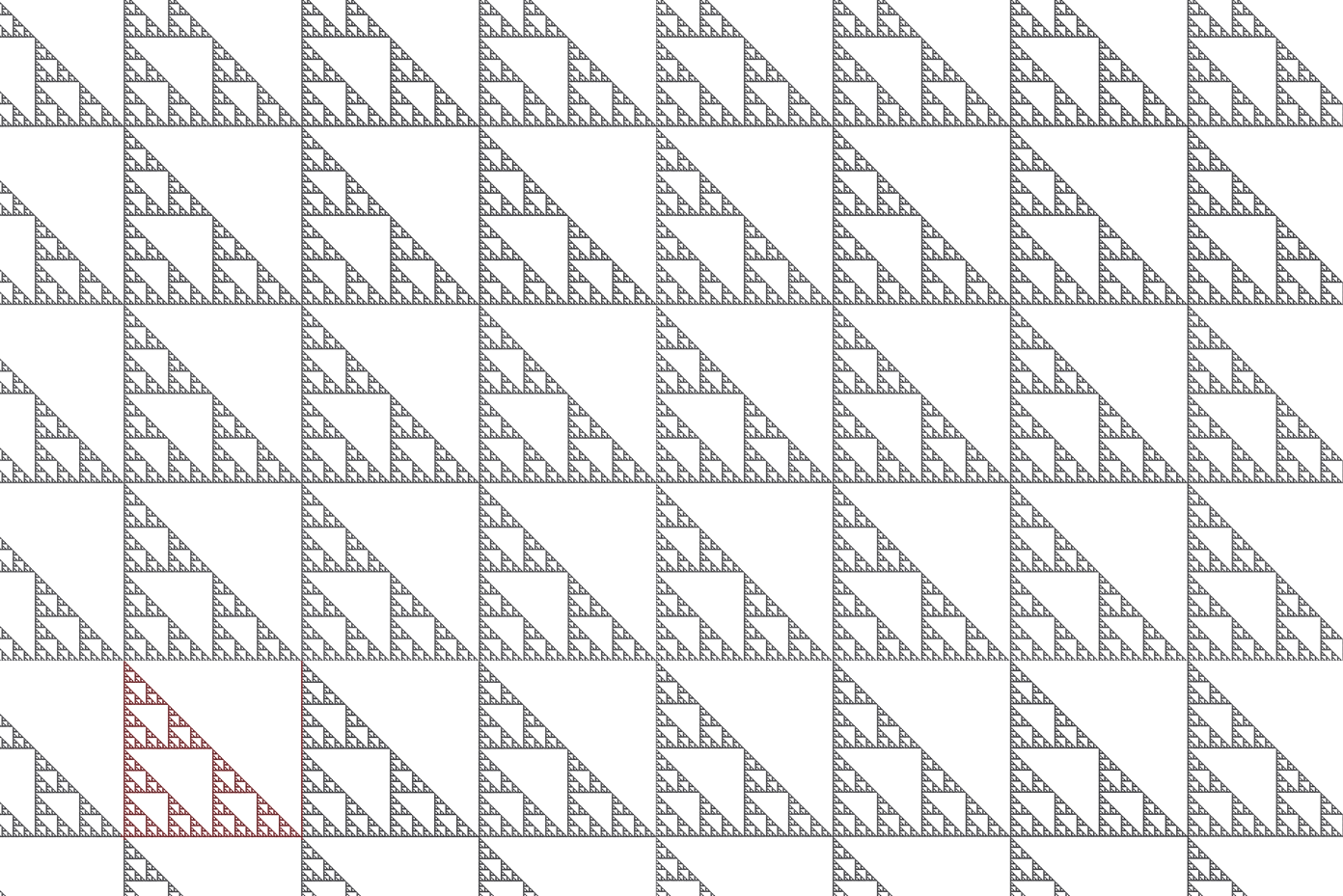}%
\caption{An example of a fast basin of an attractor of an IFS. The attractor
of the IFS, a right-angle Sierpinski triangle, is shown in red.}%
\label{manifold1}%
\end{figure}
Here the fast basin is
\[
\left\{  z\in\mathbb{C}:\exists k\in\mathbb{N},(i_{1},i_{2},...,i_{k}%
)\in\{1,2,3\}^{k},f_{i_{1}}\circ f_{i_{2}}\circ...\circ f_{i_{k}}(z)\in
A\right\}  .
\]
In this and other cases, the fast basin is uniquely defined by the attractor
$A$, provided that functions of the IFS, the $f_{i}$s, are appropriately
analytic. The fast basin is related to the attractor $A$ analogously to the
way that analytic continuation of a function is related to its (multivariable)
Taylor series expansion about a point. The attractor plays the role of the
power series, and the fast basin plays the role of the analytic continuation.
This analogy can be made precise for analytic fractal interpolation functions,
as explained in \cite{BarnsleyVince}.

An example, illustrating the relationship between fast basin and analytic
continuation, is provided by the IFS
\[
\mathcal{F}_{1}\mathcal{=}\{\mathbb{C\times C};f_{+1+i},f_{-1+i}%
,f_{+1-i},f_{-1-i}\}
\]
where
\begin{align*}
f_{\pm1-i}%
\begin{pmatrix}
z\\
w
\end{pmatrix}
&  =%
\begin{pmatrix}
\frac{1}{2} & 0\\
\frac{\pm1-i}{2} & \frac{1}{4}%
\end{pmatrix}%
\begin{pmatrix}
z\\
w
\end{pmatrix}
+%
\begin{pmatrix}
\frac{\pm-i}{2}\\
\frac{\mp i}{2}%
\end{pmatrix}
;\\
f_{\pm1+i}%
\begin{pmatrix}
z\\
w
\end{pmatrix}
&  =%
\begin{pmatrix}
\frac{1}{2} & 0\\
\frac{\pm1+i}{2} & \frac{1}{4}%
\end{pmatrix}%
\begin{pmatrix}
z\\
w
\end{pmatrix}
+%
\begin{pmatrix}
\frac{\pm1+i}{2}\\
\frac{\pm i}{2}%
\end{pmatrix}
.
\end{align*}
The unique attractor $A_{1}$ of $\mathcal{F}_{1}$ is the graph of $z\mapsto
z^{2}$ over the square $-1\leq\operatorname{Re}z,\operatorname{Im}z\leq+1$.
The fast basin of $A_{1}$ is the manifold
\[
\left\{  (z,z^{2})\in\mathbb{C\times C}:z\in\mathbb{C}\right\}  .
\]
In other words, the fast basin of $A_{1}$ (w.r.t. the IFS $\mathcal{F}_{1}$)
is the Riemann surface for $z\mapsto z^{2}$ over $\mathbb{C}$. This manifold
can be characterized as the set of points $(z_{0},w_{0})\in$ $\mathbb{C\times
C}$ such that there exists $f:\mathbb{C\times C\rightarrow C\times C}$ of the
form $f(z,w)=(az+b,F(z,w))$, where $F(z,w)$ is holomorphic and invertible,
$F(\mathbb{C\times C)=C}$, $a,b\in\mathbb{C}$, with the properties (i)
$f(A_{1})\subset A_{1}$ and $\ $(ii) $f(z_{0},w_{0})\in A_{1}$. The fast
basin, in this case, is independent of the analytic IFS that is used to
generate it, modulo some natural conditions.

Fast basins are distinct from the "macrofractals" discussed in \cite{Banakh}
which include (the union of two isometric copies of) the "macro-Cantor" set%
\[
\left\{  \sum_{k=0}^{\infty}2x_{k}3^{k}:(x_{k})\in\{0,1\}^{\mathbb{N}%
}\right\}  \text{.}%
\]
The latter is an asymptotic counterpart of the Cantor set, see
\cite{BanakhZarichnyi}, and also \cite{Dranishnikov}. For a contractive IFS on
a complete metric space, the "macrofractal" is the closure of the set of fixed
points of the inverse or dual IFS.

In Section \ref{defsec} we define and discuss IFSs, their (strict) attractors,
basins, continuations, and fast basins. We also mention other motivations for
this study.

In Section \ref{theoremsec} we establish how connectivity, porosity,
dimension, and possession of an empty interior, of the attractor are shared
with the fractal continuations and the fast basin of the attractor. The main
conclusions are summarized in Table~\ref{tableofinvariance}.

In Section \ref{Examples} we illustrate examples of fast basins.

In Section \ref{sec:dynamics} we consider the dynamics induced by the IFS on
the fast basin of an attractor.

In Section \ref{slowsec} define the slow basin of an attractor of an IFS, and
establish some of its basic topological properties. Roughly speaking, the slow
basin of an attractor $A$ comprises those points whose $\omega$-limit set
under the Hutchinson operator of the IFS has non empty intersection with the
attractor. It includes both the basin and the fast basin of an attractor.

\section{\label{defsec}Definitions}

For the purposes of this paper we use the following definition of an IFS.

\begin{defin}
An \textit{iterated function system (IFS)} is a topological space $X$ together
with a finite set of homeomorphisms $w_{i}:X\rightarrow X$, $i=1,2,\dots,N$.
\end{defin}

We use the notation
\begin{equation}
\mathcal{W}=\{X;w_{1},w_{2},...,w_{N}\}
\end{equation}
to denote an IFS. Other more general definitions of an IFS are used in the
literature; for example, the collection of functions in the IFS may be
infinite, see for example \cite{Wicks,K,Arbieto}, or the functions may
themselves be set-valued \cite{vrscay, lasota}. However, throughout this
paper, $N$ is a finite positive integer and, except where otherwise stated,
$X$ is a complete metric space.

The \textit{Hutchinson operator} $W:\mathcal{K}(X)\rightarrow\mathcal{K}(X)$
is defined on the family of nonempty compact sets $S\in\mathcal{K}(X)$ by
$W(S):=\bigcup_{i=1}^{N}w_{i}(S)$. The $k$-fold composition of $W$ is denoted
by $W^{k}$ with the convention that $W^{0}$ means the identity map. By the
inverse image of $S\subset X$ under $W$ we understand the set
\[
W^{-1}(S)=\{x\in X:W(x)\cap S\neq\emptyset\}.
\]
This is the large counter-image employed in set-valued analysis, cf.
\cite{AliprantisBorder}. Obviously
\[
W^{-1}(S)=\bigcup_{n=1}^{N}w_{n}^{-1}(S),
\]
where $w_{n}^{-1}(S)$ is the image of $S$ under the inverse map $w_{n}^{-1}$
or, equivalently, the counter-image of $S$ under $w_{n}$.

Throughout we assume that the IFS $\mathcal{W}$ possesses a strict attractor,
$A$. Following \cite{BarnsleyVince-Projective} we recall that a closed set
$A\subset X$ is a \textit{strict attractor} of $\mathcal{W}$, when there
exists an open set $B\supset A$ such that $W^{k}(S)\rightarrow A$ for every
nonempty compact $S\subset B$, where the set-convergence is meant in the sense
of Hausdorff. The maximal open set $B(A)$ with the above property is called
the \textit{basin} of the attractor $A$ (with respect to $\mathcal{W}$). From
the definition it follows that $A$ is compact, nonempty and invariant; that
is
\[
A\in\mathcal{K}(X)\text{, and }W(A)=A\text{;}%
\]
see for example \cite{BarnsleyLesniak-Continuity,Arbieto}.

We tend to omit the adjective 'strict' and refer to $A$ briefly as an
attractor. However, the reader should be aware that there are other
definitions of the notion of an attractor, see for example \cite{RPGT}. It is
widely known that strict attractors occur in contractive systems; see
\cite{FractalsEver,Falconer,Edgar} for discussion of contractive systems as
described in Hutchinson's original paper \cite{Hutchinson}, and see
\cite{BarnsleyDemko,Hata,Wicks,Mate,AndresFiser} for discussion of more
general contractive systems. But an IFS which is not contractive can possess a
strict attractor, see \cite{Kameyama,BarnsleyVince-Projective,Vince} and also
the following simple example.

\begin{example}
Let $X$ be a compactum and $h:X\rightarrow X$ be a homeomorphism such that $X$
is a minimal invariant set, i.e., if $S\in\mathcal{K}(X)$ and $h(S)=S$, then
$S=X$. By a virtue of the Birkhoff's minimal invariant set theorem (see
\cite{Gottschalk} and references therein) we know that the forward orbit of
any point in $X$ under $h$ is dense in $X$,
\[
\forall_{x_{0}\in X}\;\overline{\{h^{k}(x_{0}):k\geq0\}}=X.
\]
A canonical situation of this kind arises for an irrational rotation of the
circle. The IFS $\mathcal{W=}\{X;e,h\}$, where $e$ is the identity map on $X$,
has strict attractor $A=X$ with $B(A)=X$. But this IFS is noncontractive and
cannot be remetrized into a system of (weak) contractions. The identity map
makes remetrization to a contractive system impossible. Moreover, this is an
example of an IFS where the attractor is not point-fibred in the sense of
Kieninger (cf. \cite{K}) and is not topologically self-similar in the sense of
Kameyama (cf. \cite{Kameyama}). To see that $X$ is the unique strict attractor
of $\mathcal{W}$ we note the following. First, for all $x_{0}\in X$,
\[
W^{k}(\{x_{0}\})=\{h^{m}(x_{0}):0\leq m\leq k\}\rightarrow\overline
{\{h^{m}(x_{0}):m\geq0\}}=X.
\]
Second, for a general $S\in\mathcal{K}(X)$ we have $W^{k}(S)\rightarrow X$,
because $W^{k}(\{x_{0}\})\subset W^{k}(S)\subset X$ for arbitrary $x_{0}\in S$.
\end{example}

For a finite word $(\theta_{1},\theta_{2},\ldots,\theta_{k})\in\{1,{\ldots
},N\}^{k}$ we define
\[
w_{\theta_{1}..\theta_{k}}:=w_{\theta_{1}}\circ\ldots\circ w_{\theta_{k}%
}\text{ and }w_{\varnothing}=e\text{,}%
\]
where $\varnothing$ is the empty (zero-length) word and $e$ is the identity
map. Also given an infinite word $\vartheta=(\theta_{1},\ldots,\theta
_{k},\ldots)\in\{1,\ldots,N\}^{\infty}$ we write ${\vartheta}|k:=(\theta
_{1},\ldots,\theta_{k})\in\{1,\ldots,N\}^{k}$, and we define ${\vartheta
}|0:=\varnothing$. Similarly we write
\[
w_{{\vartheta}|k}^{-1}:=w_{\theta_{1}}^{-1}\circ\ldots\circ w_{\theta_{k}%
}^{-1}\text{ and }w_{{\vartheta}|k}^{-1}(S)=\{x\in X:w_{\theta_{k}}\circ
\ldots\circ w_{\theta_{1}}(x)\in S\}
\]
for all $S\subset X$.

In papers concerning the foundations of IFS theory, and also in dynamical
systems theory, basins of attractors are much studied. One reason that basins
of attractors are of interest is because they provide examples of sets which
are not only simple to describe, either in terms of an algorithm or by
specifying the IFS, but also geometrically complicated. The following two
examples illustrate this point.

\begin{example}
It follows from the work of Vince \cite{Vince} that the basin of an attractor
of a M\"{o}bius IFS may itself be the complement of an attractor of a
different M\"{o}bius IFS.
\end{example}

\begin{example}
The IFS
\[
\mathcal{W}_{\lambda}=\{\widehat{\mathbb{C}};w(z)=z^{2}-\lambda\},
\]
where $\lambda\in\mathbb{C}$ and $\widehat{\mathbb{C}}=\mathbb{C\cup
\{\infty\}}$, comprises a discrete dynamical system on the Riemann sphere. For
$\lambda\in(-0.25,0.75),$ $\mathcal{W}_{\lambda}$ possesses the attractor
$A=\{Z_{0}=0.5+\sqrt{1+4\lambda}/2\}$ with basin $B(A)$ which is a simply
connected domain bounded by a Jordan curve. In fact, the boundary of $B(A)$,
the Jordan curve, is a Julia set. Milnor has illustrated a related example,
\cite[Figure 1a on p.3-3]{Milnor}. Easy-to-use interactive software that
illustrates basins of attractors of $\mathcal{W}_{\lambda}$ is freely
available, see for example \cite{iPad}.
\end{example}

In this paper we draw attention to the set $\hat{B}$ of those initial
conditions $x_{0}\in X$ such that some orbit $x_{k}=w_{\theta_{1}..\theta_{k}%
}(x_{0})$ intersects the attractor $A$ after a finite number of steps. This
set is of interest in the following contexts: (i) analysis on fractals, in
connection with "fractafolds" and "fractal blow-ups" \cite{Strichartz}; (ii)
in connection with a generalization of the notion of analytic continuation, as
discussed in Section \ref{introsec}; (iii) in connection with a general
framework for understanding fractal tiling \cite{tiling}; (iv) in connection
with the chaos game algorithm for computing approximations to attractors; (v)
in connection with extending fractal transformations (between attractors of
pairs of IFSs) to transformations between basins of attractors, \cite{tiling}.

What is the relationship between $\hat{B}$ and $B(A)$? The examples in
Section~\ref{Examples} show that, despite our first impression, there is no
direct relation between $\hat{B}$ and $B(A)$ in general.

\begin{defin}
A \textit{fast basin} of the IFS $\mathcal{W}$ with the attractor $A$ is the
following set
\[
\hat{B} = \{x\in X: \exists_{k{\geq}0}\; W^{k}(x)\cap A\neq\emptyset\}.
\]

\end{defin}

\begin{defin}
A \textit{fractal continuation} of the attractor $A$ along the infinite word
$\vartheta=(\theta_{1},\theta_{2},\ldots)$ is the ascending union
\begin{equation}
\hat{B}(\vartheta)=\bigcup_{k{\geq}0}w_{{\vartheta}|k}^{-1}(A).
\label{eq:fractinuation}%
\end{equation}

\end{defin}

The following observations about the inverse images of the iterations of $W$
clarify and simplify further use of the definitions of the fast basin and
fractal continuations.

\begin{lem}
For $S\subset X$

\begin{enumerate}
\item[(i)] $\underbrace{W^{-1}(\ldots(W^{-1}}_{k}(S)\ldots) = {(W^{k})}%
^{-1}(S)$, simply denoted from now on by $W^{-k}(S)$;

\item[(ii)] $W^{-k}(S) = \bigcup_{(\theta_{1},\ldots,\theta_{k}) \in
\{1,\ldots,N\}^{k}}\; w_{\theta_{1}..\theta_{k}}^{-1}(S)$.
\end{enumerate}
\end{lem}

\begin{prop}
\label{th:representation} There hold the following representations:

\begin{enumerate}
\item[(i)] $\hat{B}(\vartheta) = \{x\in X: \exists_{k\geq0}\;\; w_{\theta
|k}(x) \in A\}$;

\item[(ii)] $\hat{B} = \{x\in X: \exists_{k\geq0}\; \exists_{(\theta
_{1},\ldots,\theta_{k}) \in\{1,\ldots,N\}^{k}}\;\; w_{\theta_{1}..\theta_{k}%
}(x) \in A\}$;

\item[(iii)] $\hat{B} = \bigcup_{k=0}^{\infty} W^{-k}(A) = \bigcup
_{\vartheta\in\{1,{\ldots},N\}^{\infty}} \hat{B}(\vartheta)$.
\end{enumerate}
\end{prop}

Particularly the descriptive formula
\begin{equation}
\label{eq:countabledescript}\hat{B} = \bigcup_{k\geq0}\; \bigcup_{(\theta
_{1},\theta_{2},\ldots,\theta_{k})\in\{1,\ldots,N\}^{k}}\; w_{\theta
_{1}..\theta_{k}}^{-1}(A)
\end{equation}
which follows from combining the above proposition and lemma, shall be useful.

\section{\label{theoremsec}Theorems}

This is the central section which shows that many properties of the attractor
are inherited by its fractal continuations and fast basin. We summarize
everything in the common Table~\ref{tableofinvariance} (cf. similar tables in
\cite{Engelking}).

\begin{table}[ptb]
\caption{Invariant properties}%
\label{tableofinvariance}%
\begin{tabular}
[c]{lll}%
$A$ & $\hat{B}(\vartheta)$ & $\hat{B}$\\\hline
connected & $+$ & $+$\\
pathwise connected & $+$ & $+$\\
boundary set (empty interior) & $+$ & $+$\\
$\sigma$-porous & $+$ $(*)$ & $+$ $(*)$\\
topological (covering) dim $= m$ & $+$ & $+$\\
fractal (Hausdorff) dim $= s$ & $+$ $(*)$ & $+$ $(*)$\\\hline
\multicolumn{3}{r}{$+$ inherits the property from $A$}\\
\multicolumn{3}{r}{$(*)$ provided $w_{i}$ are bi-Lipschitz}%
\end{tabular}
\end{table}

We start with the property of dimension. By $dim$ we mean either the
topological (\v{C}ech-Lebesgue covering) dimension or the fractal
(Hausdorff-Besicovitch) dimension.

\begin{thm}
[Sum theorem for dimension \cite{Engelking,Falconer}]Let $X$ be a complete
metric space and $\{F_{k}\}_{k}$ be a countable family of closed subsets of
$X$. Then
\[
dim\left(  \bigcup_{k} F_{k}\right)  = \sup_{k} dim(F_{k}).
\]

\end{thm}

\begin{thm}
[Invariance theorem for dimension \cite{Engelking,Falconer}]Let $h:X\to X$ be
a homeomorphism and $E\subset X$. Then
\[
dim(h(E)) = dim(E),
\]
provided additionally that $w$ is bi-Lipschitz in the case when $dim$ is the
Hausdorff dimension.
\end{thm}

\begin{thm}
\label{th:dim} For the attractor $A$ of $\mathcal{W}$, its fractal
continuation $\hat{B}(\vartheta)$ along $\vartheta\in\{1,\ldots,N\}^{\infty}$
and the fast basin $\hat{B}$ the formulas

\begin{enumerate}
\item[(i)] $dim(\hat{B}(\vartheta)) = dim(A)$,

\item[(ii)] $dim(\hat{B}) = dim(A)$
\end{enumerate}

hold true, provided additionally that $w_{i}$ are bi-Lipschitz in the case of
the Hausdorff dimension $dim$.
\end{thm}

\begin{proof}
Observe that $w_{\vartheta|k}^{-1}$ is a homeomorphism (bi-Lipschitz where
necessary), so by the invariance of dimension $dim(A) = dim(w_{\vartheta
|k}^{-1}(A))$. Putting now $F_{k} := w_{\vartheta|k}^{-1}(A))$ for $k\geq0$ in
the sum theorem yields (i) due to \eqref{eq:fractinuation}.

In the same way the sum and invariance theorems give (ii) due to \eqref{eq:countabledescript}.
\end{proof}

Our reasoning considered both cases $\hat{B}$ and $\hat{B}(\vartheta)$
separately because of the two obstacles:

\begin{enumerate}
\item[(a)] the topological dimension in general metric spaces lacks
monotonicity, so one cannot exploit the inclusion $A\subset\hat{B}%
(\vartheta)\subset\hat{B}$;

\item[(b)] the union representation in Proposition \ref{th:representation}
(iii) need not be countable.
\end{enumerate}

Next we study the connectedness of fractal continuations and the fast basin.
First note that in the realm of locally compact metric spaces (the class where
are build many geometric models) zero-dimensionality is equivalent to
hereditary disconnectedness (i.e., lack of connected nonsingleton subsets, a
property weaker than the celebrated extreme disconnectedness of the Cantor
set). By the de Groot theorem such spaces admit ultrametrization so they have
a tree-like structure. Therefore, we get for free from Theorem \ref{th:dim},
that whenever the attractor is hereditarily disconnected (has a tree-like
structure) the same holds true for its fractal continuations and the whole
fast basin. (The above discussion followed \cite{Engelking}).

\begin{thm}
[Sum theorem for connectedness \cite{Engelking}]Let $X$ be a metric space and
$\{C_{j}\}_{j}$ be a family of connected (respectively pathwise connected)
subsets of $X$ such that $\bigcap_{j} C_{j}\neq\emptyset$. Then the union
$\bigcup_{j} C_{j}$ is again connected (respectively pathwise connected).
\end{thm}

\begin{thm}
[Invariance theorem for connectedness \cite{Engelking}]Let $h:X\to X$ be a
homeomorphism and $E\subset X$ be a connected (respectively pathwise
connected) set. Then $h(E)$ is again connected (respectively pathwise connected).
\end{thm}

\begin{thm}
If the attractor $A$ of $\mathcal{W}$ is (pathwise) connected, then both its
fractal continuation $\hat{B}(\vartheta)$ along $\vartheta\in\{1,\ldots
,N\}^{\infty}$ and the fast basin $\hat{B}$ are (pathwise) connected.
\end{thm}

\begin{proof}
One needs only to remind that according to \eqref{eq:fractinuation} and
\eqref{eq:countabledescript}, the sets $\hat{B}$ and $\hat{B}(\vartheta)$ are
build from the homeomorphic copies $w_{\theta_{1}..\theta_{k}}^{-1}(A) \supset
A$ of $A$, $\theta_{1},\ldots,\theta_{k}\in\{1,\ldots,N\}$, $k\geq0$, and the
intersection of anyhow chosen collection of these copies is nonempty. Hence
the invariance and the sum theorem for connectedness give the desired conclusion.
\end{proof}

In the end of this section we study how thin{\slash}thick attractors affect
their continuations and fast basins. Let us recall (cf.
\cite{Lucchetti,Zajicek,RPGT}) that a set $S\subset X$ is \textit{porous}
provided
\[
\exists_{0< \lambda<1}\; \exists_{r_{0}>0}\; \forall_{s\in S}\; \forall
_{0<r<r_{0}}\; \exists_{x\in X}\; N_{{\lambda} r}\{x\} \subset N_{r}%
\{s\}\setminus S,
\]
where $N_{r}\{x\} := \{y\in X: d(y,x)<r\}$ stands for an open ball. A $\sigma
$-porous set is a countable union of porous sets.

\begin{thm}
If the attractor $A$ of $\mathcal{W}$ is thin in one of the following senses:

\begin{enumerate}
\item[(i)] $A$ is $\sigma$-porous,

\item[(ii)] $A$ is boundary (i.e., the interior $int A=\emptyset$),
\end{enumerate}

then both its fractal continuation $\hat{B}(\vartheta)$ along $\vartheta
\in\{1,\ldots,N\}^{\infty}$ and the fast basin $\hat{B}$ are thin in the same
sense; provided in the case of $\sigma$-porosity that $w_{i}$ are bi-Lipschitz.
\end{thm}

\begin{proof}
For (a) it is enough to note that the image of a porous set via bi-Lipschitz
homeomorphism (onto the whole space $X$) is again porous.

Part (b) needs the Baire category theorem \cite{AliprantisBorder}: a countable
intersection of dense open sets is dense. We consider only $\hat{B}$ which
contains the case $\hat{B}(\vartheta)\subset\hat{B}$. Keep in mind that
$h:=w_{\theta_{1}..\theta_{k}}^{-1}$ is a homeomorphism for $\theta_{1}%
,\ldots,\theta_{k}\in\{1,\ldots,N\}$, $k\geq0$. The sets $X\setminus h(A)$ are
open, because $A$ is closed. They are dense, because
\[
\overline{X\setminus h(A)} = \overline{h(X\setminus A)} = h(\overline
{X\setminus A}) = h(X\setminus int(A)) =h(X)=X
\]
due to the assumption that $int(A)=\emptyset$. By \eqref{eq:countabledescript}
we obtain that
\[
X\setminus\hat{B} = \bigcap_{k\geq0}\; \bigcap_{(\theta_{1},\theta_{2}%
,\ldots,\theta_{k})\in\{1,\ldots,N\}^{k}}\; X\setminus w_{\theta_{1}%
..\theta_{k}}^{-1}(A)
\]
is the countable intersection of dense open sets. Hence it is dense, what
means exactly that $int(\hat{B})=\emptyset$.
\end{proof}

To obtain more properties of the fast basin related to the properties of the
attractor one needs to understand better how the ascending sets $W^{-k}(A)$
sit in the fast basin
\[
\hat{B} = \bigcup_{k\geq0} W^{-k}(A).
\]
This unavoidably leads to studying the dynamics on the fast basin
(Section~\ref{sec:dynamics}).

\section{Examples}

\label{Examples}

We illustrate parts of fast basins.

Figure \ref{fastbasin010} shows part of the fast basin and, in \ red, the
attractor of the IFS
\begin{equation}
\mathtt{\mathcal{F=\{}}\mathbb{R}^{2};\frac{1}{2}(-y,x),\frac{1}%
{2}(-y+2,x),\frac{1}{2}(-y,x+2)\}\text{.} \label{IFS01}%
\end{equation}
(Here we write $(-0.5y,0.5x)$ to mean the function on $\mathbb{R}^{2}$ such
that $(x,y)\mapsto(-0.5y,0.5x)$.) The viewing window is $-6.2\leq x,y\leq6.3$.%

\begin{figure}[ptb]%
\centering
\includegraphics[
natheight=13.652900in,
natwidth=13.652900in,
height=3.4421in,
width=3.4421in
]%
{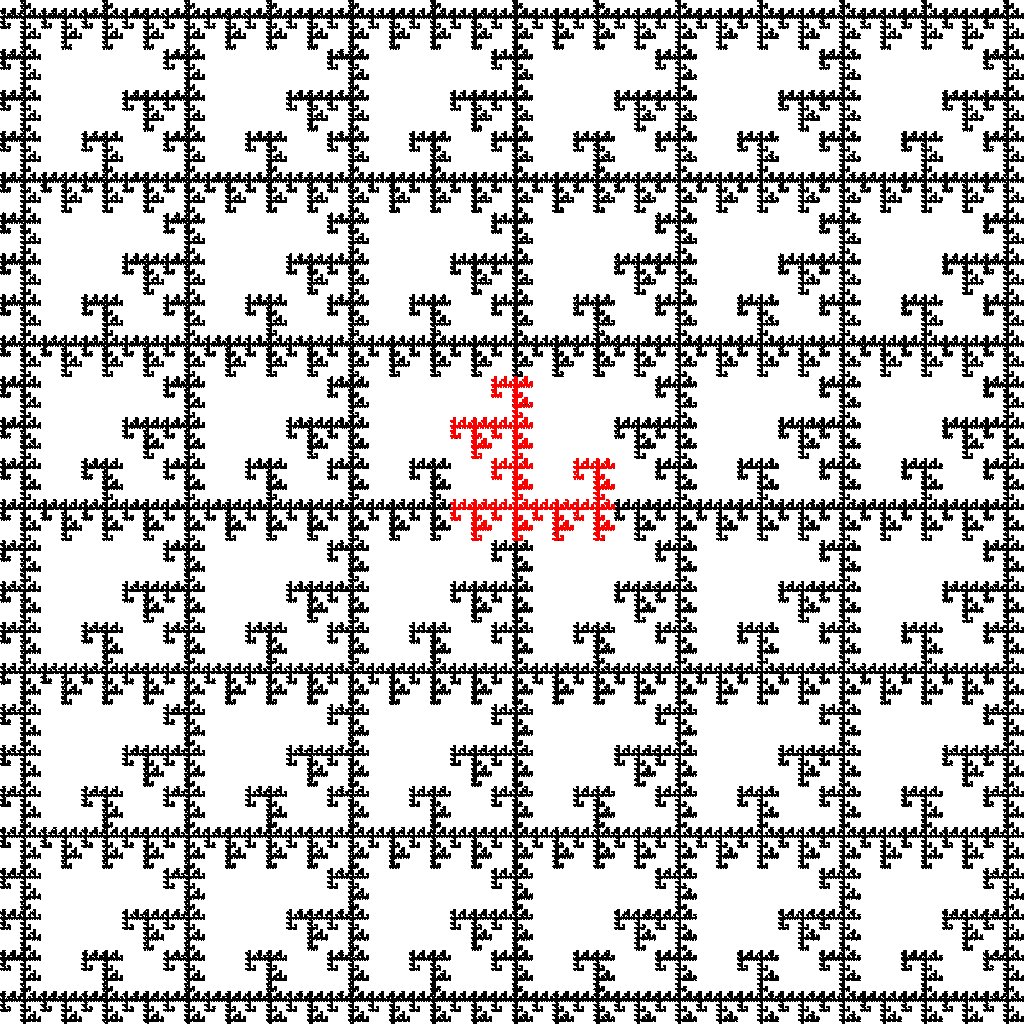}%
\caption{Attractor of the IFS (\ref{IFS01}) and the part of its fast basin
within the viewing window. See text.}%
\label{fastbasin010}%
\end{figure}
Figure \ref{fastbasin010zones} illustrates a larger region of same fast basin,
and colours encode the "generations" of the fast basin: points in the region
in black arrive in four iterations (but in no less number), and so on, as
explained in the caption.%
\begin{figure}[ptb]%
\centering
\includegraphics[
natheight=14.221800in,
natwidth=14.221800in,
height=3.5409in,
width=3.5409in
]%
{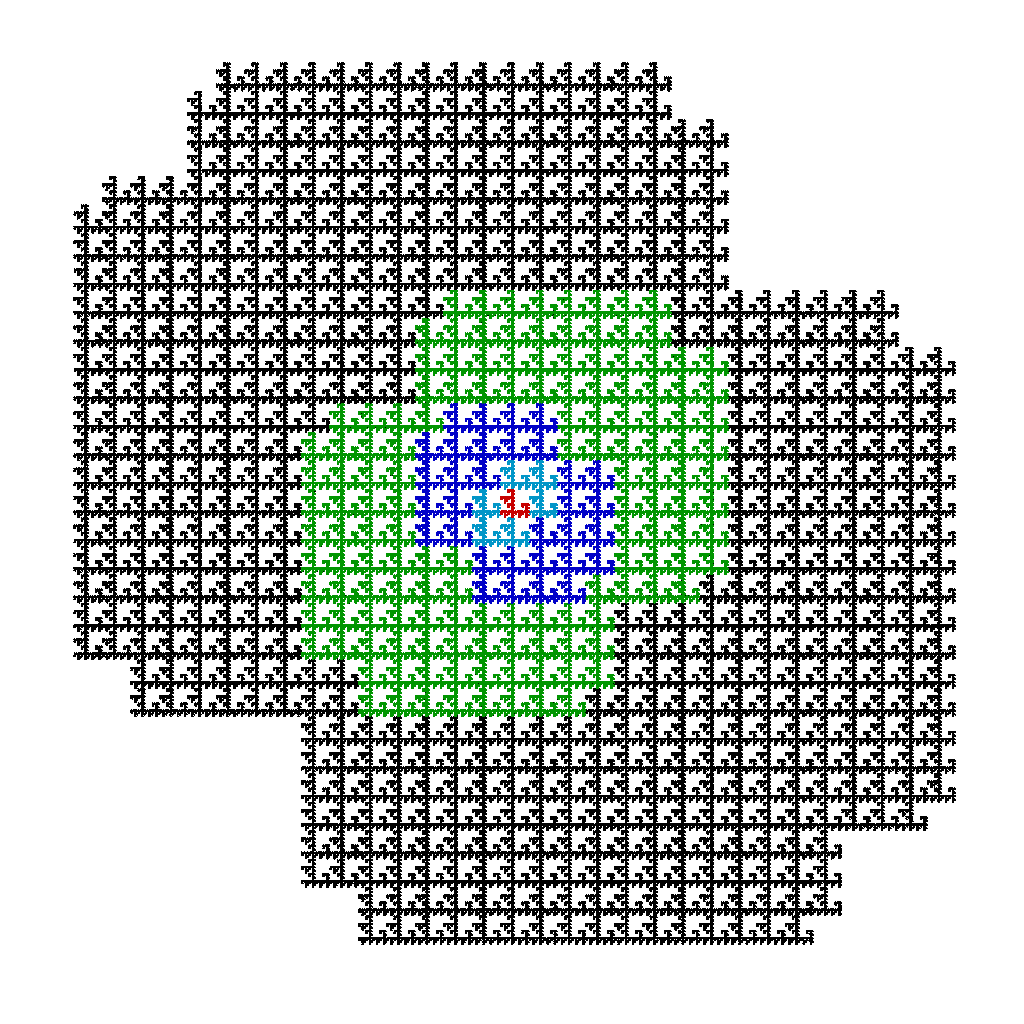}%
\caption{Generations of the fast basin that is also shown in Figure
\ref{fastbasin010}. Black points take four (and no less) iterations to arrive
on the attractor (red), green points take three (and no less) iterations, dark
blue points take two (and no less) and light blue take one iteration.}%
\label{fastbasin010zones}%
\end{figure}

Fast basin of the Kigami triangle (i.e., the Sierpi\'{n}ski triangle in
harmonic coordinates according to \cite{Kigami}).
\begin{figure}[ptb]%
\centering
\includegraphics[
natheight=13.760000in,
natwidth=18.560301in,
height=2.6451in,
width=3.5575in
]%
{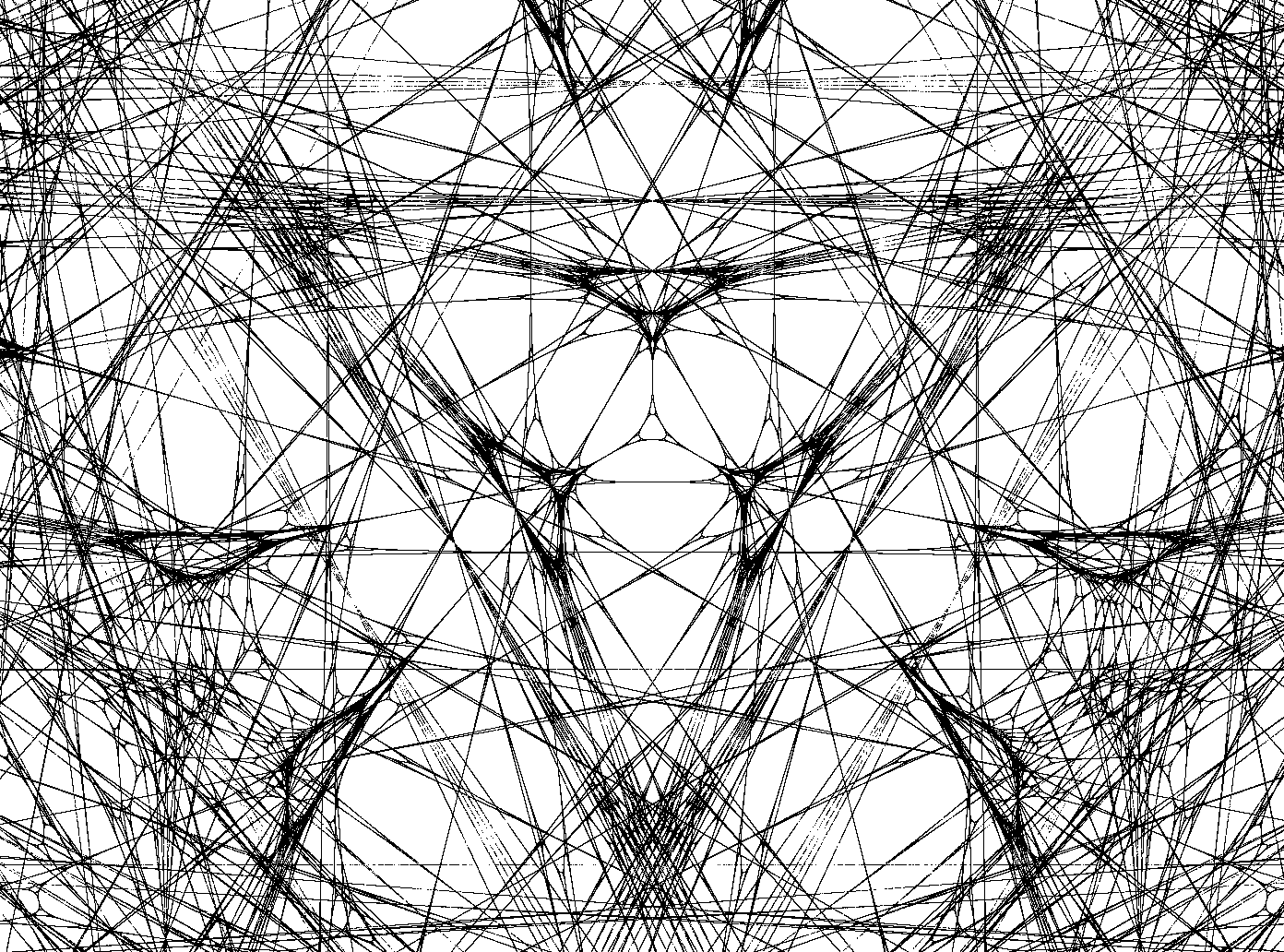}%
\caption{Illustration of part of the fast basin of the Kigami triangle. See
text}%
\label{k3}%
\end{figure}
Figure \ref{k3} shows part of the fast basin of Kigami triangle, involving
affines rather than similitudes. The IFS is
\[
\{\mathbb{R}^{2};\frac{1}{5}(2x+y,x+2y),\frac{1}{5}(3x+2,-x+y+1),\frac{1}%
{5}(x-y+1,3y+2)\}.
\]
The Kigami triangle itself is in the center of the image. In Figure
\ref{fastbasin016zones} the colours index the "generations" of the fast basin.%
\begin{figure}[ptb]%
\centering
\includegraphics[
natheight=14.221800in,
natwidth=14.221800in,
height=3.0156in,
width=3.0156in
]%
{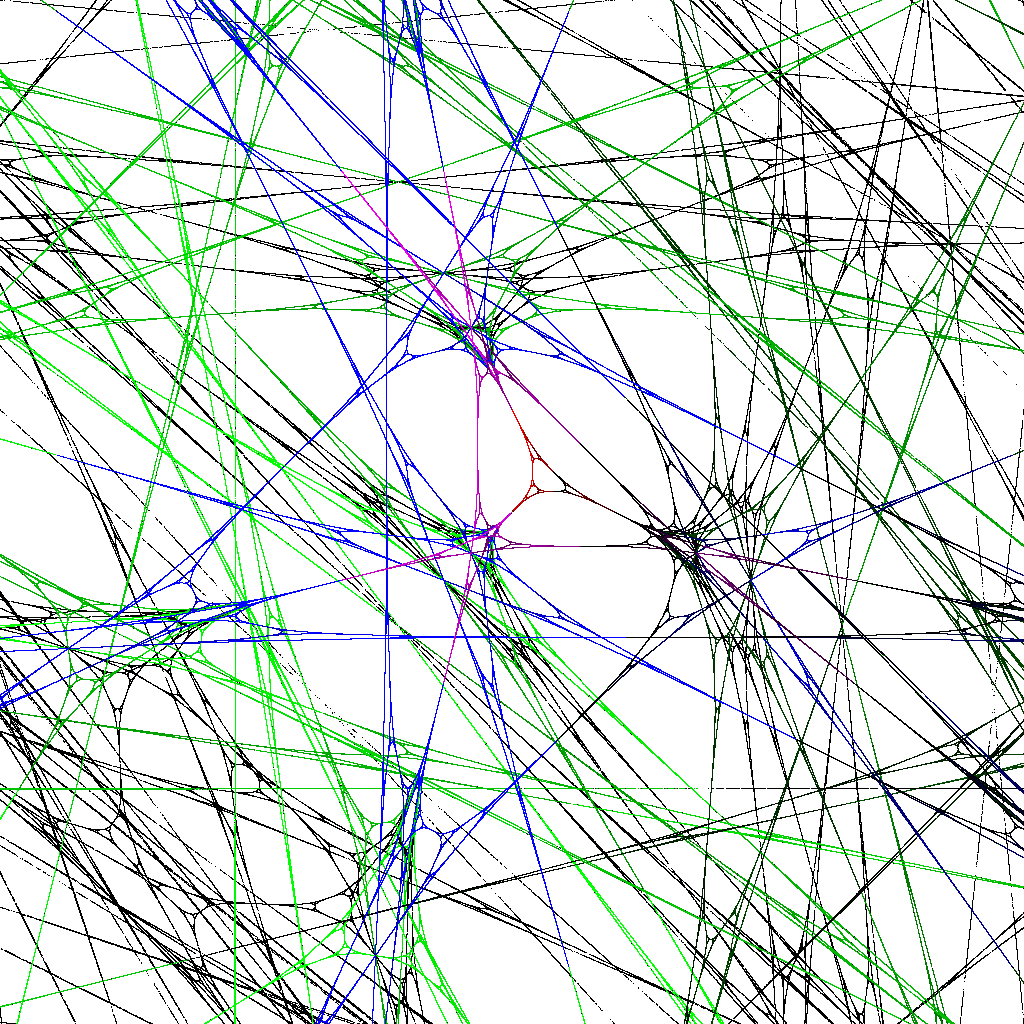}%
\caption{Some of the different "generations", in different colours, of the
fast basin of the Kigami triangle attractor. See also Figure \ref{k3}.}%
\label{fastbasin016zones}%
\end{figure}

A beautiful example of a fast basin is illustrated in Figure
\ref{fastbasin023}. The IFS in this case is%
\[
\{\mathbb{R}^{2};\frac{1}{2}(x,y+1),\frac{1}{2}(-y+1,-x+1),\frac{1}%
{2}(y+1,-x+1)\}.
\]%
\begin{figure}[ptb]%
\centering
\includegraphics[
natheight=13.333300in,
natwidth=13.333300in,
height=3.4105in,
width=3.4105in
]%
{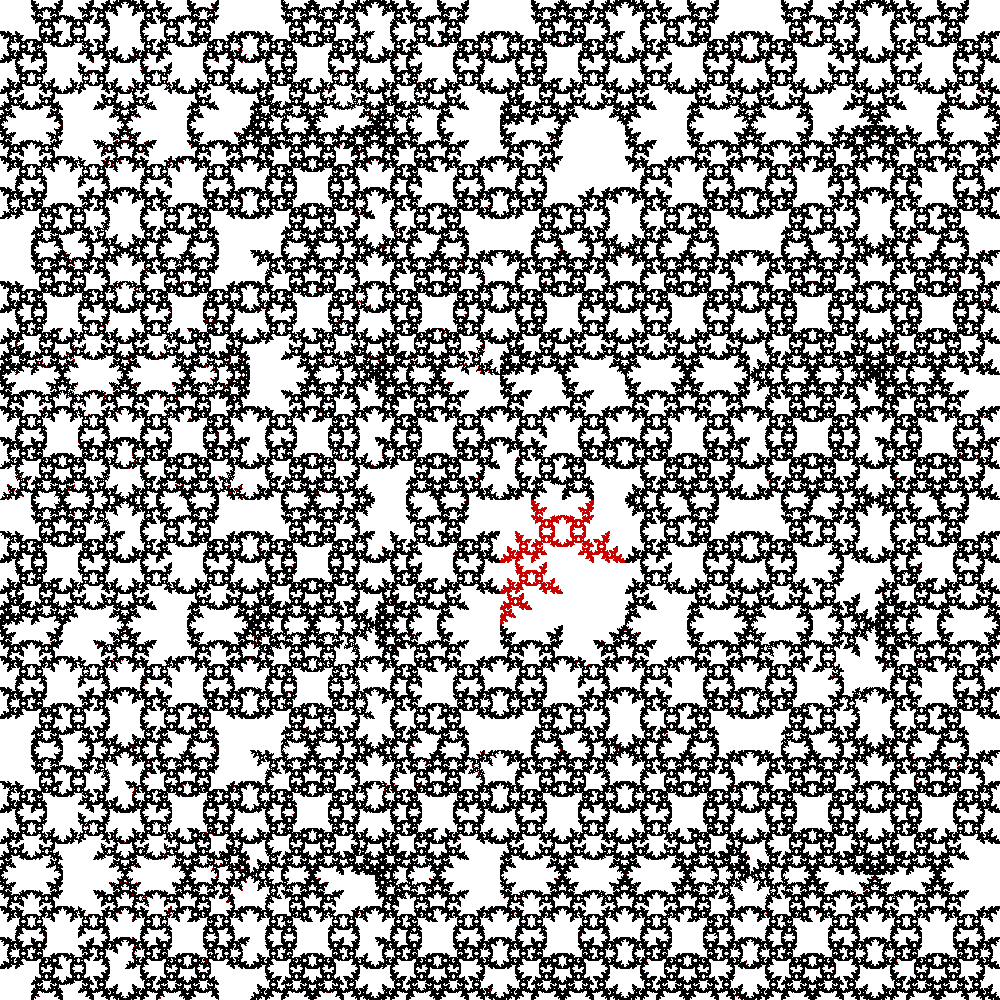}%
\caption{Part of the fast basin associated with the attractor shown in red.
See text.}%
\label{fastbasin023}%
\end{figure}
The intricate geometrical complexity of this fast basin contrasts with the
algebraic simplicity of the IFS.%
%

\begin{example}
(Fast basin reaching outside the basin.) Let $X=\mathbb{R}\cup\{\infty\}$ be
the one-point compactification of the real line. Define $w_{1}(x):=x/2$ for
$x\neq\infty$, $w_{1}(\infty):=\infty$, and
\[
w_{2}(x):=\left\{
\begin{array}
[c]{ll}%
\frac{x+3}{6-2x}, & x\not \in \{3,\infty\},\\
\infty, & x=3,\\
-\frac{1}{2}, & x=\infty.\\
&
\end{array}
\right.
\]
Then $A=[0,1]$ and $\hat{B}\not \subset B(A)$.

Since $w_{1}([0,1]) = [0,1/2]$, $w_{2}([0,1]) = [1/2,1]$, we have that
$A:=[0,1] = W(A)$ is indeed the only candidate for attractor.

The map $w_{1}$ has $0$ as an (exponential) attractor, so we study the
behavior of $w_{2}$. Firstly $w_{2}(x)\geq w_{2}(3/2)=3/2$ for $3>x\geq3/2$,
so $[3/2,3)\subset X\setminus B(A)$. Moreover $w_{1}(\infty)=w_{2}%
(3)=\infty\not \in B(A)$, so $3\not \in B(A)$. Secondly $w_{2}(x)<x$ for
$x\in(1,3/2)$. Thirdly $w_{2}(1)=1$ and the derivative $0<w_{2}^{\prime
}(x)\leq3/4$ for $x\leq1$. Therefore
\[
B(A) = \mathbb{R}\setminus[3/2,\infty).
\]

Now observe
\[
w_{1}^{-n}([3/2,3)) = 2^{n}\cdot[3/2,3),
\]
\[
[3/2,\infty) = \bigcup_{n=0}^{\infty} 2^{n}\cdot[3/2,3) \subset X\setminus
B(A),
\]
\[
w_{1}^{n+1}(2^{n}\cdot3) = w_{1}(3/2) = 3/4\in A,
\]
\[
\{3\cdot2^{n}: n\geq1\} \subset\hat{B}\setminus B(A).
\]
Altogether $\hat{B}\not \subset B(A)$.
\end{example}

\section{The dynamics on the fast basin}

\label{sec:dynamics}

The first observation expresses how it is "easy" to escape the fast basin
$\hat{B}$ : the orbit of any point not on the fast basin does not meet the
fast basin.

\begin{prop}
If $x\not \in \hat{B}$, then $W(x)\cap\hat{B}=\emptyset$.
\end{prop}

\begin{proof}
Ad absurdum suppose that some $y=w_{\theta}(x)\in W(x)$, $\theta\in
\{1,\ldots,N\}$, falls into $\hat{B}$. Then
\[
w_{\theta\theta_{1}..\theta_{k}}(x) = w_{\theta_{1}..\theta_{k}}(y)\in A,
\]
which leads to $x\in\hat{B}$.
\end{proof}

The next result explains when the fast basin is trivial in terms of the action
of the IFS on the attractor.

\begin{prop}
\label{th:criterion} The following are equivalent:

\begin{enumerate}
\item[(i)] $\hat{B} \neq A$,

\item[(ii)] $w_{i}^{-1}(A)\neq A$ for some $i=1,\ldots,N$,

\item[(iii)] $w_{i}(A)\neq A$ for some $i=1,\ldots,N$.
\end{enumerate}
\end{prop}

\begin{proof}
Recall that $W^{-1}(A) = \bigcup_{i=1}^{N}\; w_{i}^{-1}(A)$ and $\hat{B}=
\bigcup_{k\geq0}\; W^{-k}(A)$. If $\hat{B}=A$, then $\hat{B}\supset
W^{-1}(A)=A$. Conversely, if $W^{-1}(A)=A$, then $W^{-k}(A)=A$ for all
$k\geq0$, so $\hat{B}=A$.

By the invariance of $A$ for all $i=1,\ldots,N$, $w_{i}(A)\subset A$, and so
$w_{i}^{-1}(A)\supset A$. Thus $W^{-1}(A)\neq A$ is possible if and only if
$w_{i}^{-1}(A)\neq A$ for some $i=1,\ldots,N$. This establishes that (i) and
(ii) are equivalent.

Equivalence of (ii) and (iii) is a consequence of the bijectivity of $w_{i}$.
\end{proof}

The assumption that the maps $w_{i}$ are homeomorphisms onto the whole space
was crucial in the criterion for the fast basin to be nontrivial (cf. the fast
basin of the Julia set).

\begin{example}
Let $X:= [0,1]\times[1/2,\infty)\subset\mathbb{R}^{2}$ be endowed with the
metric induced from the Euclidean distance in the plane. Define $w_{1}(x,y):=
(x/2,\sqrt{y})$, $w_{2}(x,y):= (x/2+1/2,\sqrt{y})$, for $(x,y)\in X$. The
attractor of $(X;w_{1},w_{2})$ is $A=[0,1]\times\{1\}$. Maps $w_{i}$ are
homeomorphisms onto their images $w_{i}: X\to w_{i}(X)$. However
$w_{i}(A)\subsetneq A$ despite $w_{i}^{-1}(A)=A$ for $i=1,2$. The fast basin
$\hat{B}=A$ is trivial here.
\end{example}

From Proposition \ref{th:criterion} we have also an improvement upon
Proposition \ref{th:representation} (iii).

\begin{prop}
Let $I:= \{i=1,\ldots,N : w_{i}(A)\subsetneq A\}$ and $\theta_{1}%
,\ldots,\theta_{k}\in\{1,\ldots,N\}$. Then

\begin{enumerate}
\item[(a)] $w_{\theta_{1}..\theta_{k}\sigma}^{-1}(A) \supsetneq w_{\theta
_{1}..\theta_{k}}^{-1}(A)$ exactly when $\sigma\in I$,

\item[(b)] $\hat{B}=\bigcup_{\vartheta\in I^{\infty}}\hat{B}(\vartheta)\cup
A$.\texttt{ }
\end{enumerate}
\end{prop}

Further we consider the reverse dynamics $\mathcal{W}^{-1}:=(\hat{B}%
;w_{1}^{-1},\ldots,w_{N}^{-1})$ on the fast basin. We assume that all
$w_{i}^{-1}$ are $L>1$ expansive, i.e.,
\[
\forall{y_{1},y_{2}\in X},d(w_{i}^{-1}(y_{1}),w_{i}^{-1}(y_{2}))\geq L\cdot
d(y_{1},y_{2}).
\]
Moreover we assume that $A$, an attractor of $(X;w_{1},\ldots,w_{N})$, is
connected. The reversed dynamics $y_{n}=w_{\theta_{1}...\theta_{n}}^{-1}%
(y_{0})$ on $\hat{B}$ has two opposite components.

\begin{enumerate}
\item[(a)] Outside big enough disks everything on the fast basin is
"immediately taken away"; there is no "wandering around". This is precisely
stated in Proposition~\ref{expans}.

\item[(b)] Given a disk $D\supset A$ around the attractor, we have that the
reverse trajectories $y_{n}$ starting at attractor, $y_{0}\in A$, can have
arbitrarily large escape from disk times, namely
\begin{equation}
\label{eq:infinitetime}\sup_{y_{1}\neq y_{0}\in A} t(y_{0}) =\infty,
\end{equation}
where
\[
t(y_{0}):= \sup\{n: \forall_{m\leq n} y_{m}\in D\}.
\]

\end{enumerate}

From the above observations it follows that the whole intricate structure of
the fast basin is produced nearby the attractor and then flushed into the
whole space (look at Figures \ref{k3} and \ref{fastbasin016zones}).

\begin{lem}
\label{encloseball} For $x_{0}\in X$, $1<\tilde{L}<L$ there exists $r_{0}>0$
s.t. for all $r\geq r_{0}$, $i=1,{\ldots},N$
\[
w_{i}^{-1}(X\setminus D(x_{0},r)) \subset X\setminus D(x_{0},\tilde{L}r).
\]

\end{lem}

\begin{proof}
Find $\rho$ s.t. $D(x_{0},\rho) \supset\{w_{i}^{-1}(x_{0}): i=1,{\ldots},N\}$.
Next assign $r_{0}:= \rho{\slash}(L-\tilde{L})$. Thus $\tilde{L}r\leq Lr-\rho$
for $r\geq r_{0}$ and one readily verifies that
\[
d(w_{i}^{-1}(x),x_{0}) \geq|d(w_{i}^{-1}(x),w_{i}^{-1}(x_{0})) - d(x_{0}%
,w_{i}^{-1}(x_{0}))| > Lr-\rho
\]
for $x\not \in D(x_{0},r)$, $i=1,\ldots,N$.
\end{proof}

\begin{prop}
\label{expans} Let $x_{0}\in X$, $1<\tilde{L}<L$. Then there exists $r_{0}>0$
s.t. $d(y_{n+m},x_{0})\geq\tilde{L}^{n}\cdot d(y_{m},x_{0})$ whenever
$d(y_{m},x_{0})>r_{0}$
\end{prop}

Finally we establish \eqref{eq:infinitetime}. Define $\delta(a):=
d(a,w_{\theta_{1}}^{-1}(a))$ for $a\in A$ and fixed $\theta_{1}$. Put
$\Delta(a):= r-d(a,x_{0})$. This controls exits from the disk; namely $y\in
D(x_{0},r)$ as long as $d(y,a)\leq\Delta(a)$. Now we track the distance of the
reverse trajectory $(y_{n})$ from its starting point $y_{0}:=a\in A$:
\[
d(y_{n},y_{0}) \leq\sum_{k=0}^{n-1} L^{k}\cdot d(y_{1},y_{0}) \leq n\cdot
L^{n}\delta(a).
\]
To keep the first $n$ points $y_{1},\ldots,y_{n}$ in $D(x_{0},r)$ it is
therefore enough to take $a\in A$ such that
\[
\delta(a) < \frac{\Delta(a)}{n\cdot L^{n}}.
\]
Indeed, $\Delta(a)$ does not vary too much $\|\Delta(a)\| \geq const >0$ for
$a$ close to $conv\left(  \bigcup_{i=1}^{N} Fix(w_{i})\right)  $, and $A\ni
a\mapsto\delta(a)$ is a continuous function on the connected set, and
$\delta^{-1}(0) = Fix(w_{\theta_{1}})$, so one can find $a\in A$,
$y_{1}=w_{\theta_{1}}^{-1}(y_{0})\neq y_{0}=a$, with sufficiently small
$\delta(a)$.

\section{\label{slowsec}Slow basin}

In the present section we review some basic notation from the theory of
hyperspaces as, unlike in the discussion so far, we need to deal with this
formalism in a direct way.

Let $x\in X$, $A\subset X$, $r>0$. We shall write
\[
d(x,A) := \inf_{a\in A} d(x,a)
\]
for the distance from $x$ to $A$,
\[
N_{r}A := \{x\in X: d(x,A)<r\}
\]
for the $r$-neighbourhood of $A$, and
\[
D_{r}A := \{x\in X: d(x,A)\leq r\}.
\]
for the $r$-dilation of $A$.

We say that a sequence $(x_{n})_{n=1}^{\infty}$ \textit{converges to the set}
$A\subset X$, denoted $x_{n}\to A$, whenever $d(x_{n},A) {\to}0$.

\begin{defin}
A \textit{slow basin} of the IFS $\mathcal{W}$ with the attractor $A$ is the
following set
\[
\tilde{B} = \{x\in X: \exists_{\vartheta\in\{1,\ldots,N\}^{\infty}}\;
w_{\vartheta|n}(x)\to A\}.
\]

\end{defin}

\begin{prop}
\label{th:slowbascontainsall} The slow basin $\tilde{B}$ contains both the
fast basin $\hat{B}$ and the basin $B(A)$ of the attractor $A$.
\end{prop}

\begin{proof}
It is enough to check that $B(A)\subset\tilde{B}$. Take $x\in B(A)$ and any
$\vartheta\in\{1,\ldots,N\}^{\infty}$. Then
\[
d(w_{\vartheta|n}(x),A)\leq\sup\{d(y,A) : y\in W^{n}(x)\} \leq d_{H}%
(W^{n}(x),A){\to}0
\]
where $d_{H}$ denotes the Hausdorff distance (\cite{AliprantisBorder,Beer}).
\end{proof}

\begin{lem}
\label{th:invslowbas} The slow basin is backward (i.e., negatively) invariant:
if $x\in\tilde{B}$, then $W^{-1}(x)\subset\tilde{B}$.
\end{lem}

\begin{proof}
Let $x\in\tilde{B}$. Then $w_{\vartheta|n}(x)\to A$ for some $\vartheta
\in\{1,\ldots,N\}^{\infty}$. Every $y\in W^{-1}(x)$ provides representation
$x=w_{\sigma}(y)$ with some $\sigma\in\{1,\ldots,N\}$. Hence
\[
w_{\sigma\theta_{1}..\theta_{n}}(y) = w_{\theta_{1}..\theta_{n}}(x)\to A,
\]
so $x\in\tilde{B}$.
\end{proof}

\begin{thm}
We have the following representations of the slow basin

\begin{enumerate}
\item[(i)] $\tilde{B} = \bigcup_{k\geq0} W^{-k}(B(A))$,

\item[(ii)] there exists $r_{0}>0$ such that for $0<r<r_{0}$
\[
\tilde{B} = \bigcup_{k\geq0} W^{-k}(N_{r}A) = \bigcup_{k\geq0} W^{-k}%
(D_{r}A).
\]

\end{enumerate}

In particular, the slow basin is an open set.
\end{thm}

\begin{proof}
Since $B(A)\supset A$ is an open neighbourhood of a compact set, there exists
$r_{0}$ such that
\[
A\subset N_{r}A \subset D_{r}A\subset N_{r_{0}}A\subset B(A)
\]
for $r<r_{0}$. Obviously
\[
\tilde{B} \supset\bigcup_{k\geq0} W^{-k}(B(A)) \supset\bigcup_{k\geq0}
W^{-k}(D_{r}A) \supset\bigcup_{k\geq0} W^{-k}(N_{r}A),
\]
$0<r<r_{0}$, due to Lemma \ref{th:invslowbas} and Proposition
\ref{th:slowbascontainsall}.

Now let $x\in\tilde{B}$. Then $w_{\theta_{1}..\theta_{n}}(x)\to A$ for some
$(\theta_{1},\ldots,\theta_{n},\ldots)\in\{1,\ldots,N\}^{\infty}$. From the
definition of convergence there exists $k$ such that $w_{\theta_{1}%
..\theta_{n}}(x)\in N_{r}A$, i.e., $x\in W^{-k}(N_{r}A)$. Therefore $\tilde
{B}\subset\bigcup_{k\geq0} W^{-k}(N_{r}A)$.
\end{proof}

\begin{thm}
If $X$ is a space with the property that its open balls are path connected
sets, and the attractor $A$ of $\mathcal{W}= (X;w_{1},\ldots,w_{N})$ is
connected, then the slow basin $\tilde{B}$ is a path connected set.
\end{thm}

\begin{proof}
Since $A$ is connected it is chainable: for each $\varepsilon>0$ and
$a_{0},a\in A$ there exists $\{a_{1},\ldots,a_{m}\}\subset A$ with
$d(a_{i-1},a_{i})<\varepsilon$, $i=1,\ldots,m$, $a_{m}=a$ (Exercise 3.2.8 (b)
p.90 \cite{Beer}). Thus $N_{r} A = \bigcup_{a\in A} N_{r} \{a\}$ is path
connected as a connected union of path connected balls.

The sets $w_{i}^{-1}(N_{r}A)$ are homeomorphic images of path connected
$N_{r}A$ and form connected union $W^{-1}(N_{r}A) = \bigcup_{i=1}^{N}
w_{i}^{-1}(N_{r}A)$, because $w_{i}^{-1}(N_{r}A) \supset w_{i}^{-1}(A)\supset
A$.

Inductively all $W^{-k}(N_{r}A)$ are path connected, $k\geq0$, and
$W^{-k}(N_{r}A)\supset A$. This gives path connectedness of $\tilde{B}%
=\bigcup_{k\geq0}W^{-k}(N_{r}A)$.
\end{proof}

\end{document}